\documentclass[11pt]{amsart}

\usepackage{amssymb,amsmath,color,hyperref}
\usepackage[mathscr]{eucal}

\theoremstyle{plain}
\newtheorem{thm}{Theorem}[section]
\newtheorem{theorem}[thm]{Theorem}
\newtheorem{lem}[thm]{Lemma}
\newtheorem{lemma}[thm]{Lemma}

\newtheorem{prop}[thm]{Proposition}

\theoremstyle{definition}

\theoremstyle{remark}

\newtheorem{remark}{Remark}
\newtheorem*{remark*}{Remark}

\numberwithin{equation}{section}

        \newcommand{\field}[1]{{\mathbb{#1}}}
        \newcommand{\NN}{\field{N}}
        \newcommand{\ZZ}{\field{Z}}
        
        \newcommand{\RR}{\field{R}}
        \newcommand{\CC}{\field{C}}

\begin{document}

\title[Exponential localization for eigensections]{Exponential localization for eigensections of the Bochner-Schr\"odinger operator}

\author[Y. A. Kordyukov]{Yuri A. Kordyukov}
\address{Institute of Mathematics, Ufa Federal Research Centre, Russian Academy of Sciences, 112~Chernyshevsky str., 450008 Ufa, Russia} \email{yurikor@matem.anrb.ru}


%
%
\begin{abstract}
We study asymptotic spectral properties of the Bochner-Schr\"odinger operator $H_{p}=\frac 1p\Delta^{L^p\otimes E}+V$ on high tensor powers of a Hermitian line bundle $L$ twisted by a Hermitian vector bundle $E$ on a Riemannian manifold $X$ of bounded geometry under assumption that the curvature form of $L$ is non-degenerate. At an arbitrary point $x_0$ of $X$ the operator $H_p$ can be approximated by a model operator $\mathcal H^{(x_0)}$, which is a Schr\"odinger operator with constant magnetic field. For large $p$, the spectrum of $H_p$  asymptotically coincides, up to order $p^{-1/4}$, with the union of the spectra of the model operators $\mathcal H^{(x_0)}$ over $X$. We show that, if the union of the spectra of $\mathcal H^{(x_0)}$ over the complement of a compact subset of $X$ has a gap, then the spectrum of $H_{p}$ in the gap is discrete and the corresponding eigensections decay exponentially away the compact subset.
\end{abstract}

\date{\today}

 \maketitle
\section{Introduction}
\subsection{The setting}
Let $(X,g)$ be a smooth Riemannian manifold of dimension $d$ without boundary, $(L,h^L)$ a Hermitian line bundle on $X$ with a Hermitian connection $\nabla^L$ and $(E,h^E)$ a Hermitian vector bundle of rank $r$ on $X$ with a Hermitian connection $\nabla^E$. We suppose that $(X, g)$ is  a manifold of bounded geometry and $L$ and $E$ have bounded geometry. This means that the curvatures $R^{TX}$, $R^L$ and $R^E$ of the Levi-Civita connection $\nabla^{TX}$, connections $\nabla^L$ and $\nabla^E$, respectively, and their derivatives of any order are uniformly bounded on $X$ in the norm induced by $g$, $h^L$ and $h^E$, and the injectivity radius $r_X$ of $(X, g)$ is positive.

For any $p\in \NN$, let $L^p:=L^{\otimes p}$ be the $p$th tensor power of $L$ and let
\[
\nabla^{L^p\otimes E}: {C}^\infty(X,L^p\otimes E)\to
{C}^\infty(X, T^*X \otimes L^p\otimes E)
\] 
be the Hermitian connection on $L^p\otimes E$ induced by $\nabla^{L}$ and $\nabla^E$. Consider the induced Bochner Laplacian $\Delta^{L^p\otimes E}$ acting on $C^\infty(X,L^p\otimes E)$ by
\begin{equation}\label{e:def-Bochner}
\Delta^{L^p\otimes E}=\big(\nabla^{L^p\otimes E}\big)^{\!*}\,
\nabla^{L^p\otimes E},
\end{equation} 
where $\big(\nabla^{L^p\otimes E}\big)^{\!*}: {C}^\infty(X,T^*X\otimes L^p\otimes E)\to
{C}^\infty(X,L^p\otimes E)$ is the formal adjoint of  $\nabla^{L^p\otimes E}$. Let $V\in C^\infty(X,\operatorname{End}(E))$ be a self-adjoint endomorphism of $E$. We assume that $V$ and its derivatives of any order are uniformly bounded on $X$ in the norm induced by $g$ and $h^E$. 
We study the Bochner-Schr\"odinger operator $H_p$ acting on $C^\infty(X,L^p\otimes E)$ by
\[
H_{p}=\frac 1p\Delta^{L^p\otimes E}+V. 
\] 
The operator $H_p$ is self-adjoint in the Hilbert space $L^2(X,L^p\otimes E)$ with domain  $H^2(X, L^p\otimes E)$, the second Sobolev space, see \cite{Kor91,ko-ma-ma}. We denote by $\sigma(H_p)$ its spectrum in $L^2(X,L^p\otimes E)$.

Consider the real-valued closed 2-form $\mathbf B$ (the magnetic field) given by 
\begin{equation}\label{e:def-omega}
\mathbf B=iR^L. 
\end{equation} 
We assume that $\mathbf B$ is non-degenerate. Thus, $X$ is a symplectic manifold. In particular, its dimension is even, $d=2n$, $n\in \NN$. 

For $x\in X$, let $B_x : T_xX\to T_xX$ be the skew-adjoint operator such that 
\[
\mathbf B_x(u,v)=g(B_xu,v), \quad u,v\in T_xX. 
\]
The operator $|B_x|:=(B_x^*B_x)^{1/2} : T_xX\to T_xX$ is a positive self-adjoint operator. We assume that it is uniformly positive on $X$: 
\begin{equation}\label{e:uniform-positive} 
b_0:=\inf_{x\in X}|B_x|>0.
\end{equation}

\subsection{Main results}
For an arbitrary $x_0\in X$, the model operator at $x_0$ is a second order differential operator $\mathcal H^{(x_0)}_{p}$, acting on $C^\infty(T_{x_0}X, E_{x_0})$, which is obtained from the operator $H_p$ by freezing coefficients at $x_0$. This operator was introduced by Demailly \cite{Demailly85,Demailly91}.

Consider the trivial Hermitian line bundle $L_0$ over $T_{x_0}X$ and the trivial Hermitian vector bundle $E_0$ over $T_{x_0}X$ with the fiber $E_{x_0}$. We introduce the connection 
\begin{equation}\label{e:nablaL0}
\nabla^{(x_0)}_{p}=d-ip\theta^{(x_0)}, 
\end{equation}
acting on $C^\infty(T_{x_0}X, L^p_0\otimes E_0)\cong C^\infty(T_{x_0}X, E_{x_0})$, with the connection one-form $\theta^{(x_0)}\in \Omega^1(T_{x_0}X)$ given by 
\begin{equation}\label{e:Aflat}
\theta^{(x_0)}_v(w)=\frac{1}{2}\mathbf B_{x_0}(v,w),\quad v\in T_{x_0}X, \quad w\in T_v(T_{x_0}X)\cong  T_{x_0}X. 
\end{equation}
The curvature of $\nabla^{(x_0)}_{p}$ is constant: $d\theta^{(x_0)}=\mathbf B_{x_0}$.  Denote by $\Delta^{(x_0)}_{p}$ the associated Bochner Laplacian. The model operator $\mathcal H^{(x_0)}_{p}$ acting on $C^\infty(T_{x_0}X, E_{x_0})$ is defined as 
\begin{equation}\label{e:DeltaL0p}
\mathcal H^{(x_0)}_{p}=\frac 1p\Delta^{(x_0)}_p+V(x_0).
\end{equation}

Since $B_{x_0}$ is skew-adjoint, its eigenvalues have the form $\pm i a_j(x_0), j=1,\ldots,n,$ with $a_j(x_0)>0$. By \eqref{e:uniform-positive}, $a_j(x_0)\geq b_0>0$ for any $x_0\in X$ and $j=1,\ldots,n$.  Denote by $V_\mu(x_0), \mu=1,\ldots,r$, the eigenvalues of $V(x_0)$. It is well-known that the spectrum of $\mathcal H^{(x_0)}_{p}$ is independent of $p$ and consists of eigenvalues of infinite multiplicity: 
\begin{equation}\label{e:def-Sigmax}
\sigma(\mathcal H^{(x_0)}_{p})=\Sigma_{x_0}:=\left\{\Lambda_{\mathbf k,\mu}({x_0})\,:\, \mathbf k\in\ZZ_+^n, \mu=1,\ldots,r\right\}, 
\end{equation}
where, for $\mathbf k=(k_1,\cdots,k_n)\in\ZZ_+^n$, $\mu=1,\ldots,r$ and $x_0\in X$,
\begin{equation}\label{e:def-Lambda}
\Lambda_{\mathbf k,\mu}(x_0)=\sum_{j=1}^n(2k_j+1) a_j(x_0)+V_\mu(x_0).
\end{equation}
In particular, the lowest eigenvalue of $\mathcal H^{(x_0)}_{p}$ is 
\[
\Lambda_0(x_0):=\sum_{j=1}^n a_j(x_0)+\min _\mu V_\mu(x_0). 
\]
Let $\Sigma$ be the union of the spectra of the model operators: 
\begin{equation}\label{e:def-Sigma}
\Sigma=\bigcup_{x\in X}\Sigma_x=\left\{\Lambda_\mathbf {k,\mu}(x)\,:\, \mathbf k\in\ZZ_+^n, \mu=1,\ldots,r, x\in X \right\}.
\end{equation}

\begin{theorem}[\cite{Ko22}]\label{t:spectrum}
For any $K>0$, there exists $c>0$ such that for any $p\in \NN$ the spectrum of $H_{p}$ in the interval  $[0,K]$  is  contained in the $cp^{-1/4}$-neighborhood of $\Sigma$.  
\end{theorem}

When $X$ is compact, a stronger result, with $p^{-1/2}$ instead of $p^{-1/4}$ was proved by L. Charles \cite{charles21}. This estimate seems to be optimal. 

For an interval $[a,b]$, let $\mathcal K_{[a,b]}$ be the closed subset of $X$ given by 
\[
\mathcal K_{[a,b]}=\{x\in X \,:\,  \Sigma_x\cap [a,b]\neq\emptyset\}.
\]
In other words, $x\in \mathcal K_{[a,b]}$ iff $\Lambda_{\mathbf k,\mu}(x)\in [a,b]$ for some $\mathbf k \in \mathbb Z^n_+$ and $\mu=1,\ldots,\operatorname{rank}(E)$.

By \cite[Theorem 1.5]{Bochner-trace} (see also \cite[Theorem 1.3]{charles21}), if $x_0\not \in \mathcal K_{[a,b]}$, then the Schwartz kernel of the spectral projection $E_{[a,b]}$ of the operator $H_p$ associated with $[a,b]$ satisfies 
\begin{equation} \label{e:loc2}
\left|E_{[a,b]}(x_0,x_0)\right|=\mathcal O(p^{-\infty}),\quad p\to \infty.
\end{equation}

By this theorem, if $x_0\not \in \mathcal K_{[a,b]}$, then, for any sequence $\{u_{p}\in C^\infty(X,L^p\otimes E), p\in \mathbb N\}$ of eigenfunctions of $H_p$ with the corresponding eigenvalues $\lambda_{p}$ in $[a,b]$ for any $p\in \mathbb N$, we have
\[
 |u_{p}(x_0)|=\mathcal O(p^{-\infty}), \quad p\to \infty.
\]
In other words, the essential support of the sequence $\{u_{p}, p\in \mathbb N\}$ is contained in $\mathcal K_{[a,b]}$.

The main results of the paper are the following two theorems. 

\begin{theorem}\label{t:ess-spectrum}
Assume that, for an interval $[a,b]\subset \mathbb R$, the set $\mathcal K_{[a,b]}$
is compact. Then there exists $\epsilon>0$ such that for any $p\in \NN$ the spectrum of $H_{p}$ in $[a+\epsilon p^{-1/4},b-\epsilon p^{-1/4}]$ is discrete.  
\end{theorem}

As in Theorem~\ref{t:spectrum}, the order $p^{-1/4}$ doesn't seem to be optimal and, probably, can be improved.

\begin{thm}\label{t:eigenest}
Under the assumptions of Theorem \ref{t:ess-spectrum}, for any $[a_1,b_1]\subset (a,b)$, there exist $p_0\in \mathbb N$ and $C, c>0$ such that, for any $u_p\in C^\infty(X, L^p\otimes E)\cap L^2(X, L^p\otimes E)$ such that 
\[
H_pu_p=\lambda_pu_p
\]
with $p>p_0$ and $\lambda_p\in [a_1,b_1]$, we have  
\[
\int_{\Omega} e^{2c\sqrt{p} d (x,\mathcal K_{[a,b]})}|u_p(x)|^2dx \leq C\|u_p\|^2.
\]
\end{thm}

\subsection{Discussion}
Our study is partly motivated by the spectral theory of magnetic Schr\"odinger operators with magnetic walls. 
The magnetic wall is usually modeled by the magnetic field, whose intensity has a fast transition along a hypersurface (an interface). A typical example was introduced by Iwatsuka in \cite{Iwa}. Iwatsuka model is given by the magnetic field in $\mathbb R^2$ having positive bounded intensity $b(x_1,x_2)=b(x_2)$ that converges to two distinct constants $b_\pm$ as $x_2\to\pm \infty$. The extreme version of this model is the  magnetic field with intensity $b_->0$ for $x_2<0$ and $b_+>b_-$ for $x_2>0$, with $b_+-b_-$ large enough. Since \cite{Iwa}, there is an extensive literature devoted to the study of this class of models and its generalizations (see, for instance, \cite{AHK,DGR,DHS,FHKR,MP} and references therein). Closely related models are magnetic quantum Hall systems described by the magnetic Schr\"odinger operator with Dirichlet boundary conditions in a compact domain of the Euclidean space (this can be treated as a hard wall, or as an infinite electric potential outside of the domain; see, for instance, \cite{DeB-P,FGW,GV21} and references therein).

The analysis of such models distinguishes between edge and bulk behavior for the states associated with the Hamiltonian. We are interested in the edge states. These states are localized near the interface and generate a current along the interface, classically described by the so-called snake orbits, first introduced in \cite{RP}. The edge states exist as soon as the energy lies strictly in a gap of the set of the Landau levels. If the interface is compact, this part of the spectrum is discrete. 

The existence of the edge states was proved in \cite{DeB-P} for a constant magnetic field in a half-plane, in \cite{FGW} for a constant magnetic field in some domains in the Euclidean plane with Dirichlet boundary conditions and in \cite{DGR,DHS} for Iwatsuka models.  
In \cite{GV21}, the authors studied the edge states for the magnetic Schr\"odinger operator with Dirichlet boundary conditions in a simply-connected domain with compact boundary. In \cite{GV23}, the study of the edge states obtained for Iwatsuka models extended to the case of a general regular curve. Here the localization and propagation properties of the edge states are investigated. This study was significantly improved in \cite{FLRV24}, where the authors consider the Robin Laplacian on a smooth bounded two-dimensional domain in the presence of a constant magnetic field and obtain a uniform description of the spectrum located between the Landau levels in the semiclassical limit. In particular, they established  exponential localization near the boundary of the corresponding edge states, which was not considered in \cite{GV21,GV23}.

In our paper, we address the question of exponential localization in a very general setting of the Bochner-Schr\"odinger operator on a manifold of bounded geometry. In the above notation, one can consider the set $\mathcal K_{[a,b]}$ as an interface. Unlike \cite{GV23}, the transition of the magnetic field along the interface is not fast (we hope to discuss this case elsewhere), but the magnetic field is not constant.  Instead of the set of Landau levels associated with limiting values of the intensity of the magnetic field at infinity as in the Iwatsuka model, we are dealing with the set of local Landau levels $\Sigma_x$ assigned to each point $x$ of the manifold. Our choice of the interface ensures that the set of local Landau levels in the bulk has a gap $(a,b)$. From this point of view, we prove that if the interface $\mathcal K_{[a,b]}$ is compact, then the spectrum of the operator in $[a,b]$ is discrete, and the corresponding eigensections are edge states. Moreover, they are exponentially localized away the interface $\mathcal K_{[a,b]}$. 

Asymptotic localization of eigenfunctions of the magnetic Schr\"odinger operator associated with eigenvalues below the bottom of the essential spectrum usually follows from Agmon type estimates. But when we consider eigenvalues in gaps of the essential spectrum, such a method doesn't work. Instead, we use some weighted norm estimates for the operator. This is a slight modification of the method used in \cite{Ko22,Bochner-trace} (see also the references therein) to prove exponential localization away the diagonal for Schwartz kernels of various functions of the Bochner-Schr\"odinger operator and in \cite{FLRV24} to prove exponential localization of the boundary states of the Robin magnetic  Laplacian away the boundary, where weighted estimates for the resolvent have been used. 

The paper is organized as follows. In Section~\ref{s:discrete}, we prove Theorem \ref{t:ess-spectrum}. In Section~\ref{s:eigenest}, we prove Theorem \ref{t:eigenest}.
 
 \section{Discreteness of the spectrum} \label{s:discrete}
 
 This section is devoted to the proof of Theorem \ref{t:ess-spectrum}.
 
\subsection{Lower bound for the norm} 
 
 The following proposition plays a crucial role both in the proof of Theorem \ref{t:ess-spectrum} and in the proof of Theorem \ref{t:eigenest}.
 
 \begin{prop}\label{p:estimate}
 Let $[a,b]\subset \mathbb R$ be a bounded interval and 
 $$
 \Omega_{[a,b]}:=X\setminus \mathcal K_{[a,b]}=\{x\in X : \Sigma_x\cap [a,b]=\emptyset\}.
 $$ 
 Then there exist $C>0$ and $p_0\in\mathbb N$ such that, for any $\lambda\in (a,b)$, for any $p>p_0$ and for any $u\in C^\infty(X, L^p\otimes E)$ compactly supported in $\Omega_{[a,b]}$, we have 
 \[
 \|(H_p-\lambda)u\|\geq (d(\lambda,\Sigma)-Cp^{-1/4})\|u\|,
 \]
 where $d(\lambda,\Sigma)$ stands for the distance from $\lambda$ to $\Sigma$. 
 \end{prop}
 
 \begin{remark}
 In case $\Omega_{[a,b]}=X$, we get a new proof of Theorem~\ref{t:spectrum}.
 \end{remark}
 
The proof of Proposition \ref{p:estimate} will be given in Section~\ref{s:sppr}. First, we complete the proof of Theorem \ref{t:ess-spectrum}.
 
 \subsection{Proof of Theorem \ref{t:ess-spectrum}} 
 Theorem \ref{t:ess-spectrum} follows immediately from Proposition \ref{p:estimate} and the following manifold version of \cite[Lemma 2.1]{M88}.
 
Let $(\mathcal E,h^{\mathcal E})$ be a Hermitian vector bundle on $X$ with a Hermitian connection $\nabla^{\mathcal E}$. We suppose that $\mathcal E$ has bounded geometry. Let $V\in C^\infty_b(X,\operatorname{End}(\mathcal E))$ be a self-adjoint endomorphism. Consider the Bochner-Schr\"odinger operator $H$ acting on $C^\infty(X,\mathcal E)$ by
\[
H=\Delta^{\mathcal E}+V. 
\] 
 
 \begin{lemma}
Let $\lambda\in\mathbb R$. Suppose that there exist  $\delta>0$ and a compact subset $K\subset X$ such that
\begin{equation}
 \|(H-\lambda)u\|\geq \delta\|u\|
 \end{equation}
 for any $u\in H^2(X, \mathcal E)$ supported in $X\setminus K$. Then $\lambda\not\in \sigma_{\rm ess}(H).$
 \end{lemma}
 
 \begin{proof}
On the contrary, assume that $\lambda\in \sigma_{\rm ess}(H)$. Then there exists is an orthonormal sequence $(u_n)_{n\in\NN}$ in $L^2(X,\mathcal E)$ such that $u_n\in H^2(X,\mathcal E)$ for any $n\in\NN$ and 
\begin{equation}\label{e:weyl}
\|(H-\lambda)u_n\|\to 0, \quad n\to \infty. 
\end{equation}

There exists a sequence $\chi_m\in C^\infty_c(X)$ such that $0\leq \chi_m(x)\leq 1$ for any $x\in X$, $\chi_m\equiv 1$ in a neighborhood of $K$ and $$
|d\chi_m(x)|\leq \frac{c}{m}, \quad |\Delta\chi_m(x)|<\frac{c}{m^2}, \quad x\in X,
$$ 
where $c>0$ is independent of $m$. Such a sequence can be easily constructed, using a ``smoothed  distance''  function  $\widetilde{d}\in {C}^\infty(X\times X)$ (see, for instance, \cite[Proposition 4,1]{Kor91}), satisfying 
the following conditions:
 
(1) there is a constant $r>0$ such that
\[
\big\vert \widetilde{d}(x,y) - d (x,y)\big\vert  < \gamma,\quad x, y\in X,
\]
where $d$ stands for the distance function on $(X,g)$;

(2) for any $k>0$, there exists $C_k>0$ such that, for any multi-index 
$\beta$ with $|\beta|=k$,
\[
\big\vert \partial^\beta_{x}\widetilde{d}(x,y)\big\vert < C_{k},\quad
x, y\in X,
\]
where the derivatives are taken with respect to normal coordinates 
defined by the exponential map at $x$.

Now we assume that $K$ is contained in some open ball $B(x_0,r)\subset X$ of radius $r>0$ centered at $x_0\in X$. Take a function $\chi\in C^\infty_c(\mathbb R)$ such that $0\leq \chi(t)\leq 1$ for any $t\in \mathbb R$, $\chi(t)=1$ if $|t| \leq r+\gamma$ and put 
$$
\chi_m(x)=\chi\left(\frac 1m\widetilde{d}(x,x_0)\right),\quad x\in X.
$$
It is easy to check that the sequence $(\chi_m)_{m\in \NN}$ satisfies the desired conditions. 

Since
\[
[H, \chi_m]=-2d\chi_m\cdot\nabla^{\mathcal E} +\Delta\chi_m,
\]
we get
\[
\|[H, \chi_m]u_n\|\leq \frac{c}{m}\|\nabla^{\mathcal E}u_n\|+\frac{c}{m^2}\|u_n\|.
\]
We can estimate $\|\nabla^{\mathcal E}u_n\|$ in the following way:
\begin{multline*}
\|\nabla^{\mathcal E}u_n\|^2=(\Delta^{\mathcal E}u_n,u_n)=((H-\lambda)u_n,u_n)+((\lambda-V)u_n,u_n)\\ \leq \|(H-\lambda)u_n\|^2+C\|u_n\|^2,
\end{multline*}
and, therefore,
\[
\|\nabla^{\mathcal E}u_n\|\leq \|(H-\lambda)u_n\|+C\|u_n\|. 
\]
This gives the estimate
\begin{equation}\label{e:Hchim}
\|[H,\chi_m]u_n\|\leq \frac{C}{m}(\|(H-\lambda)u_n\|+\|u_n\|).
\end{equation}  
From the equality
\[
(H-\lambda)(1-\chi_m)u_n=(1-\chi_m)(H-\lambda)u_n-[H,\chi_m]u_n,
\]
we infer that
\[
\|(H-\lambda)(1-\chi_m)u_n\|\leq C(\|(H-\lambda)u_n\|+\frac 1m\|u_n\|).
\]
On the other hand, since $1-\chi_m$ is supported in $X\setminus K$, by assumption, we have 
\[
 \|(H-\lambda)(1-\chi_m)u_n\|\geq \delta\|(1-\chi_m)u_n\|. 
 \]
By \eqref{e:weyl}, we conclude that for any $\epsilon>0$ there exist $m$ and $N$ such that for any $n>N$, 
\begin{equation}\label{e:eps}
 \|(1-\chi_{m})u_n\|<\epsilon. 
\end{equation}  
 
Since
\[
(H-\lambda)(\chi_m u_n)=\chi_m (H-\lambda)u_n+[H,\chi_m]u_n,
\]
by \eqref{e:weyl} and \eqref{e:Hchim}, the sequence $(H-\lambda)(\chi_m u_n)_{n\in\mathbb N}$ is bounded in $L^2(X,\mathcal E)$,

By ellipticity of $H$, it follows that the sequence $(\chi_m u_n)_{n\in\mathbb N}$ is bounded in $H^2(D,\mathcal E)$, where $D\subset X$ is a regular bounded domain, which contains the support of $\chi_m$. Passing to a subsequence, we may assume that $(\chi_m u_n)_{n\in\mathbb N}$ converges to some $v\in L^2(X,\mathcal E)$. 
By \eqref{e:eps}, for any $n>N$, $\|\chi_{m}u_n\|\geq \|u_n\|- \|(1-\chi_{m})u_n\|>1-\epsilon$.  and, therefore, $\|v\|\geq 1-\epsilon$. 
 
 On the other hand, we have 
 \begin{align*}
 \|v\|^2=& \lim_{n\to \infty}\langle \chi_m u_n, \chi_m u_{n+1}\rangle \\ 
 = & \lim_{n\to \infty}\langle u_n-(1-\chi_m)u_n, u_{n+1}-(1-\chi_m) u_{n+1}\rangle \\
 = & \lim_{n\to \infty}(-\langle (1-\chi_m)u_n, u_{n+1}\rangle-\langle u_n, (1-\chi_m)u_{n+1}\rangle \\ & +\langle (1-\chi_m)u_n, (1-\chi_m) u_{n+1}\rangle) <2\epsilon +\epsilon^2. 
 \end{align*}
 We get a contradiction if we choose $\epsilon>0$ small enough. 
 \end{proof}

The rest of this section is devoted to the proof of Proposition \ref{p:estimate}. We will use some constructions introduced in \cite{Ko22}, and, therefore, we will briefly remind them, referring the interested reader to \cite{Ko22} for more details. 
 
\subsection{Approximation by the model operator}\label{s:sppr}
We construct an approximation of the operator $H_p$ by the model operator $\mathcal H^{(x_0)}_{p}$ in a sufficiently small neighborhood of an arbitrary point $x_0$. 

First, we consider some special coordinates near $x_0$. We choose an orthonormal base $\{e_j : j=1,\ldots,2n\}$ in $T_{x_0}X$ such that  
\begin{equation}\label{e:obase}
B_{x_0}e_{2k-1}=a_k(x_0)e_{2k}, \quad B_{x_0}e_{2k}=-a_k(x_0)e_{2k-1},\quad k=1,\ldots,n. 
\end{equation}
Then, for any $x_0\in X$, there exists a coordinate chart $\varkappa_{x_0} : B(0,c)\subset \RR^{2n}\stackrel{\cong}{\to} U_{x_0}=\varkappa_{x_0}(B(0,c))\subset X$ defined on the ball $B(0,c)$ of radius $c$ centered in the origin in $\RR^{2n}$ with some $c>0$, independent of $x_0$, such that 
\begin{equation}\label{e:x0}
\varkappa_{x_0}(0)=x_0,\quad (D\varkappa_{x_0})_0 (e_j)=e_j,\quad j=1,\ldots, 2n,
\end{equation} 
and $\varkappa_{x_0}^*\mathbf B$ is a constant 2-form on $B(0,c)$ given by
\begin{equation}\label{e:kappaBx0}
(\varkappa_{x_0}^*\mathbf B)_Z=\sum_{k=1}^n a_k(x_0) dZ_{2k-1}\wedge dZ_{2k}\quad Z\in B(0,c).
\end{equation}
Here, by abuse of notation, we use the same notation $\{e_j : j=1,\ldots,2n\}$ for the standard base in $\RR^{2n}$.

Moreover, for every $k\geq 0$, there exists $C_k>0$ such that, for any two charts $\varkappa_{x_\alpha} : B(0,c)\subset \RR^{2n} \stackrel{\cong}{\to} U_{x_\alpha}\subset X$ and $\varkappa_{x_\beta} : B(0,c)\subset \RR^{2n}\stackrel{\cong}{\to} U_{x_\beta}\subset X$ with $U_\alpha\cap U_\beta\neq\emptyset$, the map $\varkappa^{-1}_{x_\alpha} \circ \varkappa_{x_\beta} : \varkappa^{-1}_{x_\beta}(U_{x_\alpha}\cap U_{x_\beta})\subset \RR^{2n}\to  \RR^{2n}$ satisfies the following condition: for any multiindex $a$ with $|a|\leq k$
\begin{equation}\label{e:kp}
\|\partial^a(\varkappa^{-1}_{x_\alpha} \circ \varkappa_{x_\beta})(x)\|\leq C_k, \quad x\in  \varkappa^{-1}_{x_\beta}(U_{x_\alpha}\cap U_{x_\beta})\subset \RR^{2n}.
\end{equation}

The construction of $\varkappa_{x_0}$ is essentially the proof of the Darboux Lemma based on the well-known Moser argument. We refer the reader to \cite[Appendix]{Ko22} for more details.

It is easy to see that there exists a trivialization of the Hermitian line bundle $L$ over $U_{x_0}$:
\[
\tau^L_{x_0} : U_{x_0}\times \CC \stackrel{\cong}{\to}L\left|_{U_{x_0}}\right.,
\]
such that the connection one-form of $\nabla^L$ in this trivialization coincides with 
the one-form $\theta^{(x_0)}$ given by \eqref{e:Aflat}. We also assume that there exists a trivialization of the Hermitian bundle $E$ over $U_{x_0}$:
\[
\tau^E_{x_0} : U_{x_0}\times E_{x_0} \stackrel{\cong}{\to}E\left|_{U_{x_0}}\right.,
\]
These trivializations induce a trivialization of  $L^p\otimes E$ over $U_{x_0}$:
\[
\tau_{{x_0},p}=(\tau^L_{x_0})^p\otimes \tau^E_{x_0} : U_{x_0}\times E_{x_0} \stackrel{\cong}{\to}L^p\otimes E\left|_{U_{x_0}}\right..
\]
For any $x\in U_{x_0}$, we will write $\tau_{x_0,p}(x) : E_{x_0}\to L^p_x\otimes E_x$ for the associated linear map in the fibers. 

Let $g_{x_0}=\varkappa^*_{x_0} g$ be the Riemannian metric on $B(0,c)$ induced by the Riemannian metric $g$ on $X$. 
We introduce a map
\[
T^*_{{x_0},p} : C^\infty(X, L^p\otimes E)\to C^\infty(B(0,c), E_{x_0}),
\]
defined for $u\in C^\infty(X, L^p\otimes E)$ by 
\begin{equation}\label{e:defT}
T^*_{{x_0},p} u(Z)=|g_{x_0}(Z)|^{1/4}\tau_{{x_0},p}^{-1}(\varkappa_{x_0}(Z))[u(\varkappa_{x_0}(Z))],\quad  Z\in B(0,c).
\end{equation}

Consider the differential operator $H_p^{(x_0)}=T^*_{{x_0},p} \circ H_{p}\circ (T^*_{{x_0},p})^{-1}$ acting on $C^\infty(B(0,c), E_{x_0})$. It can be written as
\[
H_p^{(x_0)}=|g_{x_0}(Z)|^{1/4} \tau^*_{{x_0},p} \circ H_{p}\circ (\tau^*_{{x_0},p})^{-1} |g_{x_0}(Z)|^{-1/4}.
\]
Using the standard formula for the Bochner Laplacian in local coordinates, one can write
\begin{multline}\label{e:TDeltaT}
\tau^*_{{x_0},p} \circ H_{p}\circ (\tau^*_{{x_0},p})^{-1}\\ =-\frac 1p \sum_{\ell,m=1}^{2n}g_{x_0}^{\ell m}\nabla^{L^p\otimes E}_{e_\ell}\nabla^{L^p\otimes E}_{e_m}+\frac 1p \sum_{\ell=1}^{2n} \Gamma^\ell \nabla^{L^p\otimes E}_{e_\ell}+V_{x_0},
\end{multline}
where $\{e_j\}$ is the standard base in $\RR^{2n}$, $g_{x_0}^{\ell m}$ is the inverse of the matrix of $g_{x_0}$, $V_{x_0}=\tau^{E*}_{x_0} \circ V\circ (\tau^{E*}_{x_0})^{-1}\in C^\infty(B(0,c), \operatorname{End}(E_{x_0}))$ and $\Gamma^\ell\in C^\infty(B(0,c))$, $\ell=1,\ldots,2n,$ are some functions. If we denote by $\Gamma^E\in C^\infty(T(B(0,c)), \operatorname{End}(E_{x_0}))$ the connection one-form for the connection $\nabla^E$, we can write
\[
\nabla^{L^p\otimes E}_{v}=\nabla^{(x_0)}_{p,v}+\Gamma^E(v), \quad v\in T(B(0,c))=B(0,c)\times \RR^{2n}.
\]
where the connection $\nabla^{(x_0)}_{p}$ is given by \eqref{e:nablaL0}.

Then we have 
\[
|g_{x_0}|^{1/4} \nabla^{L^p\otimes E}_{v}|g_{x_0}|^{-1/4}=\nabla^{(x_0)}_{p,v}+\Gamma^E(v)-\frac 14v(\ln |g_{x_0}|).
\]
It follows that 
\begin{equation}\label{e:TDeltaT-D}
H_p^{(x_0)}=-\frac 1p \sum_{\ell,m=1}^{2n}g_{x_0}^{\ell m}\nabla^{(x_0)}_{p,e_\ell}\nabla^{(x_0)}_{p,e_m}+\frac 1p \sum_{\ell=1}^{2n} F_{\ell,{x_0}} \nabla^{(x_0)}_{p,e_\ell}+V_{x_0}+\frac 1pG_{x_0}
\end{equation}
with some $F_{\ell,{x_0}}, G_{x_0}\in C^\infty(B(0,c), \operatorname{End}(E_{x_0}))$, uniformly bounded on $x_0$. 

By \eqref{e:TDeltaT-D}, it follows that 
\begin{multline}\label{e:TDeltaT-D1}
H_p^{(x_0)}-\mathcal H^{(x_0)}_p=-\frac 1p \sum_{\ell,m=1}^{2n}(g_{x_0}^{\ell m}-\delta^{\ell m})\nabla^{(x_0)}_{p,e_\ell}\nabla^{(x_0)}_{p,e_m}\\ +\frac 1p \sum_{\ell=1}^{2n} F_{\ell,{x_0}} \nabla^{(x_0)}_{p,e_\ell}+V_{x_0}-V_{x_0}(0)+\frac 1pG_{x_0}. 
\end{multline}
By \eqref{e:x0}, we have $g_{x_0}^{\ell m}(Z)=\delta^{\ell m}, \ell,m=1,\ldots,2n$. 

\subsection{Norm estimates}
We recall some norm estimates for the model operator proved in \cite[Section 2.3]{Ko22}. We will denote by $\|\cdot\|$ the $L^2$-norm in $C^\infty_c(T_{x_0}X, E_{x_0})$. We recall that $\nabla^{(x_0)}_{p}$ stands for the connection on the trivial line bundle $L_0^p\otimes E_0$ given by \eqref{e:nablaL0} and $\Delta^{(x_0)}_{p}$ for the Bochner Laplacian on $C^\infty_c(T_{x_0}X, L_0^p\otimes E_0)\cong C^\infty_c(T_{x_0}X, E_{x_0})$ associated with this connection. 

%

Denote by $R^{(x_0)}_{p}(\lambda):=\left(\mathcal H^{(x_0)}_{p}-\lambda\right)^{-1}$, $\lambda\not\in \Sigma_{x_0}$, the resolvent of the operator $\mathcal H^{(x_0)}_{p}$. It satisfies the estimate
\begin{equation}\label{e:res1-s1}
\left\|R^{(x_0)}_{p}(\lambda)\right\|\leq d(\lambda,\Sigma_{x_0})^{-1},\quad \lambda\not\in \Sigma_{x_0},
\end{equation}
where $\|\cdot\|$ denotes the operator norm for the $L^2$-norms and $d(\lambda,\Sigma_{x_0})$ denotes the distance from $\lambda$ to $\Sigma_{x_0}$. 
Moreover, for any $K>0$, there exist $C_1>0$ and $C_2>0$ such that for any $\lambda\not\in \Sigma_{x_0}, |\lambda|<K$, 
\begin{equation*}
\left\|\frac{1}{\sqrt{p}}\nabla^{(x_0)}_pR^{(x_0)}_{p}(\lambda)\right\|\leq C_1d(\lambda,\Sigma_{x_0})^{-1},  
\end{equation*}
\begin{equation*}
\sum_{k,\ell=1}^{2n} \left\|\frac 1p\nabla^{(x_0)}_{p,e_k}\nabla^{(x_0)}_{p,e_\ell}R^{(x_0)}_{p}(\lambda) \right\| \leq C_2d(\lambda,\Sigma_{x_0})^{-1},
\end{equation*}
where $\{e_j : j=1,\ldots,2n\}$ stands for the fixed orthonormal base in $T_{x_0}X$.

As consequences, we have for any $u\in C^\infty_c(T_{x_0}X, E_{x_0})$
\begin{equation}\label{e:0H}
\left\|u\right\|\leq d(\lambda,\Sigma_{x_0})^{-1}\left\|\left(\mathcal H^{(x_0)}_{p}-\lambda\right)u\right\|,\quad 
\\ \lambda\not\in \Sigma_{x_0}.
\end{equation}

\begin{multline}\label{e:1H}
\left\|\frac{1}{\sqrt{p}}\nabla^{(x_0)}_pu\right\| \\ \leq C_1d(\lambda,\Sigma_{x_0})^{-1}\left\|\left(\mathcal H^{(x_0)}_{p}-\lambda\right)u\right\|,\quad \lambda\not\in \Sigma_{x_0}, |\lambda|<K.
\end{multline}

\begin{multline}\label{e:2H}
\sum_{k,\ell=1}^{2n} \left\|\frac 1p\nabla^{(x_0)}_{p,e_k}\nabla^{(x_0)}_{p,e_\ell}u \right\|\\ \leq C_2d(\lambda,\Sigma_{x_0})^{-1}\left\|\left(\mathcal H^{(x_0)}_{p}-\lambda\right)u\right\|,\quad \lambda\not\in \Sigma_{x_0}, |\lambda|<K.
\end{multline} 
 
 \subsection{Special covers}
Finally, we need some special covers by coordinates charts. 
For each $p\in \NN$, we consider the restrictions of the coordinates charts $\varkappa_{x_0}$ to the ball $B(0,p^{-1/4})$. One can choose an at most countable collection of coordinate charts 
\[
\varkappa_{\alpha,p}:=\varkappa_{x_{\alpha,p}}\left|_{B(0,p^{-1/4})}\right. : B(0,p^{-1/4}) \to U_{\alpha,p} := \varkappa_{\alpha,p}(B(0,p^{-1/4}))\subset X,  
\] 
with $1\leq \alpha\leq I_p$, $I_p\in \NN\cup \{\infty\}$, which cover $\Omega$ and for the cardinality of the set $\mathcal I_{p,\alpha}= \{1\leq \beta \leq I_p : U_{\alpha,p} \cap U_{\beta,p} \neq \varnothing \}$, we have 
\[
\# \mathcal I_{p,\alpha} \leq K_0, \quad 1\leq \alpha\leq I_p,
\]
with the constant $K_0$ independent of $p$.  For simplicity of notation, we will often omit $p$, writing $\varkappa_{\alpha}$, $U_\alpha$ etc. 

Choose a family of smooth functions $\{\varphi_{\alpha}=\varphi_{\alpha,p} : \RR^{2n}\to [0; 1], 1\leq \alpha\leq I_p\}$ supported on the ball $B(0,p^{-1/4})$, which gives a partition of unity on $X$ subordinate to $\{U_\alpha\}$:
\[
\sum_{\alpha=1}^{I_p}(\varphi_\alpha \circ \varkappa^{-1}_\alpha)^2 \equiv 1\ \text{on}\ \Omega,
\]
and satisfies the condition: for any $\gamma\in \ZZ^{2n}_+$, there exists $C_\gamma>0$ such that
\[
|\partial^\gamma\varphi_\alpha(Z)|<C_\gamma p^{(1/4)|\gamma|}, \quad Z\in \RR^{2n}, \quad 1\leq \alpha\leq I_p.
\] 

For every $1\leq \alpha\leq I_p$, we denote by $g_\alpha$ the induced Riemannian metric $g_{x_\alpha}$ on $B(0,p^{-1/4})$. We will use notation
\[
T^*_{\alpha}=T^*_{\alpha,p}: C^\infty(X, L^p\otimes E)\to C^\infty(B(0,p^{-1/4}), E_{x_\alpha})
\]
for the composition of the map $T^*_{x_\alpha,p}$ defined by \eqref{e:defT} with the restriction map  $C^\infty(B(0,c), E_{x_\alpha})\to  C^\infty(B(0,p^{-1/4}), E_{x_\alpha})$.

We have
\begin{equation}\label{e:Tap-unitary}
\|T^*_{\alpha}u\|^2_{L^2(B(0,p^{-1/4}), E_{x_\alpha})}=\|u\|^2_{L^2(U_\alpha,L^p\otimes E)}.
\end{equation}

 \subsection{Proof of Proposition \ref{p:estimate}}
Let $[a,b]\subset \mathbb R$ and $\Omega\subset X$ be an open domain such that for any $x\in \Omega$, $\Sigma_x\cap [a,b]=\emptyset$. 
  Let $u\in C^\infty(X, L^p\otimes E)$ be compactly supported in $\Omega$ and $\lambda\in (a,b)$. We apply the standard localization formula
\begin{equation}\label{e:local}
\|(H_p-\lambda)u\|^2
=\sum_{\alpha=1}^{I_p}\big(\|(H_p-\lambda)[(\varphi_\alpha \circ \varkappa^{-1}_\alpha)u]\|^2 -\|[H_p,\varphi_\alpha \circ \varkappa^{-1}_\alpha]u\|^2\big).
\end{equation}

Since $u$ is compactly supported in $\Omega$, we may assume that $\alpha$ belongs to the set $\mathcal I_{p,\Omega}= \{ \alpha\in I_p : U_{\alpha,p}\cap \Omega \neq \emptyset\}$. Generally speaking, for $\alpha\in \mathcal I_{p,\Omega}$, $x_\alpha$ may not belong to $\Omega$. We only know that $d(x_\alpha, \Omega)<p^{-1/4}$. Therefore, there exists $\delta>0$ such that for any $\alpha\in \mathcal I_{p,\Omega}$, we have 
\begin{equation}\label{e:delta}
\Sigma_{x_\alpha}\cap [a-\delta p^{-1/4},b+\delta p^{-1/4}]=\emptyset.
\end{equation}

For the first term in the right-hand side of \eqref{e:local}, we get
\[
\|(H_p-\lambda)[(\varphi_\alpha \circ \varkappa^{-1}_\alpha)u]\|^2=\|(H_p^{(x_\alpha)}-\lambda)[\varphi_\alpha T^*_{\alpha}u]\|^2.
\]
Since $\varphi_\alpha$ is supported on the ball $B(0,p^{-1/4})$, we have 
\[
|g^{\ell m}_\alpha(Z)-\delta^{\ell m}|\leq Cp^{-1/4}, \quad |V_\alpha(Z)-V_\alpha(0)|\leq Cp^{-1/4}, 
\] 
on the support of $\varphi_\alpha$ and therefore from \eqref{e:TDeltaT-D1} we get
\begin{multline*}
\left\|(H_p^{(x_\alpha)}-\mathcal H^{(x_\alpha)}_{p})\varphi_\alpha T^*_{\alpha}u\right\|_{L^2(\RR^{2n}, E_{x_\alpha})}\\ 
\begin{aligned}
\leq & C_1p^{-1/4}\sum_{\ell,m=1}^{2n}\left\| \frac 1p\nabla^{(x_\alpha)}_{p,e_\ell}\nabla^{(x_\alpha)}_{p,e_m}\varphi_\alpha T^*_{\alpha}u\right\|_{L^2(\RR^{2n}, E_{x_\alpha})}\\ & +C_2p^{-1/2}\sum_{\ell=1}^{2n}\left\|\frac{1}{\sqrt{p}}\nabla^{(x_\alpha)}_{p,e_\ell} \varphi_\alpha T^*_{\alpha}u\right\|_{L^2(\RR^{2n}, E_{x_\alpha})}\\  &+C_3p^{-1/4}\left\|\varphi_\alpha T^*_{\alpha}u\right\|_{L^2(\RR^{2n}, E_{x_\alpha})}.
\end{aligned}
\end{multline*}

Using  \eqref{e:0H}, \eqref{e:1H} and \eqref{e:2H}, for $\lambda\in (a+\delta p^{-1/4},b-\delta p^{-1/4})$, we have 
\begin{multline*}
\left\|(H_p^{(x_\alpha)}-\mathcal H^{(x_\alpha)}_{p})\varphi_\alpha T^*_{\alpha}u\right\|_{L^2(\RR^{2n}, E_{x_\alpha})}\\ \leq Cp^{-1/4}d(\lambda,\Sigma_{x_\alpha})^{-1}\left\|\left(\mathcal H^{(x_\alpha)}_{p}-\lambda\right)\varphi_\alpha T^*_{\alpha}u\right\|_{L^2(\RR^{2n}, E_{x_\alpha})}. 
\end{multline*}

Using the last estimate and \eqref{e:0H}, we infer that
\begin{multline*}
\|(H_p^{(x_\alpha)}-\lambda)[\varphi_\alpha T^*_{\alpha}u]\|_{L^2(\RR^{2n}, E_{x_\alpha})}\\ 
\begin{aligned}
\geq & \left\|(\mathcal H^{(x_\alpha)}_{p}-\lambda)[\varphi_\alpha T^*_{\alpha}u]\right\|_{L^2(\RR^{2n}, E_{x_\alpha})}\\ & - \left\|(H_p^{(x_\alpha)}-\mathcal H^{(x_\alpha)}_{p})\varphi_\alpha T^*_{\alpha}u\right\|_{L^2(\RR^{2n}, E_{x_\alpha})} \\ \geq & (1-Cp^{-1/4}d(\lambda,\Sigma_{x_\alpha})^{-1})
\|(\mathcal H^{(x_\alpha)}_{p}-\lambda)[\varphi_\alpha T^*_{\alpha}u]\|_{L^2(\RR^{2n}, E_{x_\alpha})} 
\\  \geq & (d(\lambda,\Sigma_{x_\alpha})-Cp^{-1/4})
\|\varphi_\alpha T^*_{\alpha}u\|_{L^2(\RR^{2n}, E_{x_\alpha})}.
\end{aligned}
\end{multline*}
Thus, for the first term in the right-hand side of \eqref{e:local}, we get
\begin{equation}\label{e:local1}
\|(H_p-\lambda)[(\varphi_\alpha \circ \varkappa^{-1}_\alpha)u]\| 
\geq (d(\lambda,\Sigma_{x_\alpha})-Cp^{-1/4}) \|(\varphi_\alpha \circ \varkappa^{-1}_\alpha)u\|.
\end{equation}

For the second term in the right-hand side of \eqref{e:local}, we write
\[
\|[H_p,\varphi_\alpha \circ \varkappa^{-1}_\alpha]u\|^2=\|[H_p^{(x_\alpha)},\varphi_\alpha]T^*_{\alpha}u\|^2
\]

Using \eqref{e:TDeltaT-D}, we compute the commutator $[H_p^{(x_\alpha)}, \varphi_\alpha]$:
\[
[H_p^{(x_\alpha)}, \varphi_\alpha] =-\frac 1p\sum_{\ell,m=1}^{2n}(2g^{\ell m}_\alpha e_\ell\varphi_\alpha\nabla^{(x_0)}_{p,e_m}+g^{\ell m}_\alpha e_\ell e_m\varphi_\alpha) +\frac 1p \sum_{\ell=1}^{2n} F_{\ell,\alpha} e_\ell\varphi_\alpha.
\]
Since $|\nabla \varphi_\alpha|<C p^{1/4}$, $|\nabla^2\varphi_\alpha|<Cp^{1/2}$, we get
\begin{align*}
\|[H_p^{(x_\alpha)},\varphi_\alpha]T^*_{\alpha}u\|_{L^2(\RR^{2n},E_{x_\alpha})} \leq & 
\frac{1}{p}\sum_{\ell,m=1}^{2n} \|{e_\ell}\varphi_\alpha\nabla^{(x_\alpha)}_{p,e_m}T^*_{\alpha}u\|_{L^2(B(0,p^{-1/4}),E_{x_\alpha})}\\ & +\frac{1}{\sqrt{p}} \|T^*_{\alpha}u\|_{L^2(B(0,p^{-1/4}),E_{x_\alpha})}.  
\end{align*}

Using that $(\nabla^{(x_\alpha)}_{p,e_m})^*=-\nabla^{(x_\alpha)}_{p,e_m}$ and $[\nabla^{(x_\alpha)}_{p,e_m}, ({e_\ell}\varphi_\alpha)^2]=2{e_\ell}\varphi_\alpha {e_\ell e_m}\varphi_\alpha$, we proceed as follows:
\begin{multline*}
\|{e_\ell}\varphi_\alpha\nabla^{(x_\alpha)}_{p,e_m}T^*_{\alpha}u\|^2_{L^2(B(0,p^{-1/4}),E_{x_\alpha})}\\
\begin{aligned}
= & ({e_\ell}\varphi_\alpha\nabla^{(x_\alpha)}_{p,e_m}T^*_{\alpha}u, {e_\ell}\varphi_\alpha\nabla^{(x_\alpha)}_{p,e_m}T^*_{\alpha}u)_{L^2(B(0,p^{-1/4}),E_{x_\alpha})}
\\
= & - (\nabla^{(x_\alpha)}_{p,e_m} ({e_\ell}\varphi_\alpha)^2 \nabla^{(x_\alpha)}_{p,e_m}T^*_{\alpha}u, T^*_{\alpha}u)_{L^2(B(0,p^{-1/4}),E_{x_\alpha})}
\\
= & - (({e_\ell}\varphi_\alpha)^2 (\nabla^{(x_\alpha)}_{p,e_m})^2T^*_{\alpha}u, T^*_{\alpha}u)_{L^2(B(0,p^{-1/4}),E_{x_\alpha})}
\\
& - 2({e_\ell}\varphi_\alpha {e_\ell e_m}\varphi_\alpha \nabla^{(x_\alpha)}_{p,e_m}T^*_{\alpha}u, T^*_{\alpha}u)_{L^2(B(0,p^{-1/4}),E_{x_\alpha})}
\end{aligned}
\end{multline*}

For the first term, using the fact that $|\nabla \varphi_\alpha|<C p^{1/4}$ and the equality
$$-\sum_{m=1}^{2n}(\nabla^{(x_\alpha)}_{p,e_m})^2=\Delta^{(x_\alpha)}=p(\mathcal H^{(x_\alpha)}_{p}-V(x_\alpha)),$$ we get 
\begin{multline*}
\sum_{\ell,m=1}^{2n}(({e_\ell}\varphi_\alpha)^2 (\nabla^{(x_\alpha)}_{p,e_m})^2T^*_{\alpha}u, T^*_{\alpha}u)_{L^2(B(0,p^{-1/4}),E_{x_\alpha})}
\\
\begin{aligned}
= & p \sum_{\ell=1}^{2n}(({e_\ell}\varphi_\alpha)^2 (\mathcal H^{(x_\alpha)}_{p}-V(x_\alpha))T^*_{\alpha}u, T^*_{\alpha}u)_{L^2(B(0,p^{-1/4}),E_{x_\alpha})}
\\
\leq & Cp^{3/2} (\|(\mathcal H^{(x_\alpha)}_{p}-\lambda)T^*_{\alpha}u\|^2_{L^2(B(0,p^{-1/4}),E_{x_\alpha})}+ \|T^*_{\alpha}u\|_{L^2(B(0,p^{-1/4}),E_{x_\alpha})}^2).
\end{aligned}
\end{multline*}
For the second term, using that $|\nabla \varphi_\alpha|<C p^{1/4}$, $|\nabla^2\varphi_\alpha|<Cp^{1/2}$, we get 
\begin{multline*}
\left|\sum_{\ell,m=1}^{2n} (2{e_\ell}\varphi_\alpha {e_\ell e_m}\varphi_\alpha \nabla^{(x_\alpha)}_{p,e_m}T^*_{\alpha}u, T^*_{\alpha}u)_{L^2(B(0,p^{-1/4}),E_{x_\alpha})}\right|
\\
\begin{aligned}
\leq & Cp^{1/2}\sum_{\ell,m=1}^{2n} \| {e_\ell}\varphi_\alpha \nabla^{(x_\alpha)}_{p,e_m}T^*_{\alpha}u\|_{L^2(B(0,p^{-1/4}),E_{x_\alpha})} \|T^*_{\alpha}u\|_{L^2(B(0,p^{-1/4}),E_{x_\alpha})}
\\
\leq & C(p^{-1/2}\sum_{\ell,m=1}^{2n} \|{e_\ell}\varphi_\alpha \nabla^{(x_\alpha)}_{p,e_m}T^*_{\alpha}u\|^2_{L^2(B(0,p^{-1/4}),E_{x_\alpha})}
\\
& +p^{3/2}\|T^*_{\alpha}u\|^2_{L^2(B(0,p^{-1/4}),E_{x_\alpha})}).
\end{aligned}
\end{multline*}
We conclude that 
\begin{multline*}
(1-Cp^{-1/2})\sum_{\ell,m=1}^{2n} \|{e_\ell}\varphi_\alpha \nabla^{(x_\alpha)}_{p,e_m}T^*_{\alpha}u\|^2_{L^2(B(0,p^{-1/4}),E_{x_\alpha})}\\
\leq Cp^{3/2} (\|(\mathcal H^{(x_\alpha)}_{p}-\lambda)T^*_{\alpha}u\|^2_{L^2(B(0,p^{-1/4}),E_{x_\alpha})}+ \|T^*_{\alpha}u\|_{L^2(B(0,p^{-1/4}),E_{x_\alpha})}^2),
\end{multline*}
and, if $p$ is large enough,
\begin{multline*}
\sum_{\ell,m=1}^{2n} \|{e_\ell}\varphi_\alpha \nabla^{(x_\alpha)}_{p,e_m}T^*_{\alpha}u\|_{L^2(B(0,p^{-1/4}),E_{x_\alpha})}\\
\leq Cp^{3/4} (\|(\mathcal H^{(x_\alpha)}_{p}-\lambda)T^*_{\alpha}u\|_{L^2(B(0,p^{-1/4}),E_{x_\alpha})}+ \|T^*_{\alpha}u\|_{L^2(B(0,p^{-1/4}),E_{x_\alpha})}).
\end{multline*}
It follows that
\begin{multline*} 
\|[H_p^{(x_\alpha)},\varphi_\alpha]T^*_{\alpha}u\|_{L^2(\RR^{2n},E_{x_\alpha})}\\ \leq Cp^{-1/4} (\|(\mathcal H^{(x_\alpha)}_{p}-\lambda)T^*_{\alpha}u\|_{L^2(B(0,p^{-1/4}),E_{x_\alpha})}+ \|T^*_{\alpha}u\|_{L^2(B(0,p^{-1/4}),E_{x_\alpha})}),
\end{multline*}
and 
\begin{multline}\label{e:local2}
\|[H_p,\varphi_\alpha \circ \varkappa^{-1}_\alpha]u\|_{L^2(\RR^{2n},E_{x_\alpha})}\\ \leq Cp^{-1/4} (\|(H_p-\lambda)[(\varphi_\alpha \circ \varkappa^{-1}_\alpha)u]\|+ \|(\varphi_\alpha \circ \varkappa^{-1}_\alpha)u\|).  
\end{multline}

Combining \eqref{e:local1} and \eqref{e:local2}, from \eqref{e:local}, assuming $d(\lambda,\Sigma)>Cp^{-1/4}$, we get 
\begin{align*}
\|(H_p-\lambda)u\|^2 =& \sum_{\alpha\in \mathcal I_{p,\Omega}}\big(\|(H_p-\lambda)[(\varphi_\alpha \circ \varkappa^{-1}_\alpha)u]\|^2 -\|[H_p,\varphi_\alpha \circ \varkappa^{-1}_\alpha]u\|^2\big)\\
\geq & \sum_{\alpha\in \mathcal I_{p,\Omega}}\big((1-C_1p^{-1/2})\|(H_p-\lambda)[(\varphi_\alpha \circ \varkappa^{-1}_\alpha)u]\|^2\\ & -C_2p^{-1/2}\|(\varphi_\alpha \circ \varkappa^{-1}_\alpha)u\|^2\big)\\
\geq & \sum_{\alpha\in \mathcal I_{p,\Omega}}\big((d(\lambda,\Sigma_{x_\alpha})-Cp^{-1/4})^2-C_3p^{-1/2})  \|(\varphi_\alpha \circ \varkappa^{-1}_\alpha)u\|^2\\
\geq & \big((d(\lambda,\Sigma)-Cp^{-1/4})^2-C_3p^{-1/2})  \|u\|^2.
\end{align*}

If $C_3\leq 0$, this completes the proof of Proposition \ref{p:estimate}. If $C_3>0$, we complete the proof as follows:
\begin{multline*}
\|(H_p-\lambda)u\|^2 \\
\begin{aligned}
\geq & \big(d(\lambda,\Sigma)-(C+C^{1/2}_3)p^{-1/4}\big)\big(d(\lambda,\Sigma)-(C-C^{1/2}_3)p^{-1/4}\big)  \|u\|^2\\
\geq & \big(d(\lambda,\Sigma)-(C+C^{1/2}_3)p^{-1/4}\big)^2  \|u\|^2
\end{aligned}
\end{multline*}
assuming $d(\lambda,\Sigma)>(C+C^{1/2}_3)p^{-1/4}$.

\section{Eigensection estimates}  \label{s:eigenest}

This section is devoted to the proof of Theorem \ref{t:eigenest}. The proof of Theorem \ref{t:eigenest} is obtained by a slight modification of the proof \cite[Proposition 2.1]{FLRV24}. Instead of resolvent estimates, we use norm estimates for the operator given by Proposition~\ref{p:estimate}. 

\subsection{Weight functions}\label{s:weighted-est}
We will use some weight functions. 
As shown in \cite[Proposition 4.1]{Kor91} (see also \cite[Section 3.1]{ko-ma-ma}), for any $p\in \mathbb N$, there exists a function $\Phi_p\in C^\infty(X)$, satisfying the following conditions:

(1) we have
\begin{equation}\label{(1.1)}
\vert \Phi_p(x) - d (x,\mathcal K_{[a,b]})\vert  < \frac{1}{\sqrt{p}}\;,
\quad x\in X, \quad p\in \NN;
\end{equation}

(2) for any $k>0$, there exists $C_k>0$ such that
\begin{equation}\label{e:chi-p}
\left(\frac{1}{\sqrt{p}}\right)^{k-1}\left|\nabla^k\Phi_p(x)\right|<C_k, \quad x\in X, \quad p\in \NN. 
\end{equation} 
 
Define a family of differential operators on $C^\infty(X,L^p\otimes E)$ by
\begin{align}\label{e:weight-operator}
H_{p,\tau}:= e^{\tau\sqrt{p} \Phi_p} H_p e^{-\tau\sqrt{p} \Phi_p},\quad p\in \NN, \quad \tau\in \RR.
\end{align}
An easy computation gives that
\begin{equation}\label{e:DpaW}
H_{p,\tau}=H_p+\frac{\tau}{\sqrt{p}} A_{p}+\tau^2B_{p},
\end{equation}
where 
\begin{equation} \label{e:ApBp}
A_{p}=-2d \Phi_p\cdot \nabla^{L^p\otimes E} +\Delta \Phi_p, \quad B_{p}=-|d \Phi_p|^2.
\end{equation}
Here, for $u \in C^\infty(X, L^p\otimes E)$, $d\Phi_p\cdot\nabla^{L^p\otimes E}u \in C^\infty(X, L^p\otimes E)$ stands for the pointwise inner product of $d\Phi_p\in C^\infty(X,T^*X)$ and $\nabla^{L^p\otimes E}u\in C^\infty(X,L^p\otimes E\otimes T^*X)$ determined by the Riemannian metric. 


\subsection{Proof of Theorem \ref{t:eigenest}} 
Suppose that $u_p\in C^\infty(X, L^p\otimes E)\cap L^2(X, L^p\otimes E)$ is such that 
\[
H_pu_p=\lambda_pu_p
\]
with some $p\in \mathbb N$ and $\lambda_p\in [a_1,b_1]\subset (a,b)$. 
Then, for $v_p=e^{\tau\sqrt{p} \Phi_p}u_p$, we have 
\begin{equation}\label{e:Hpa-lambda}
H_{p,\tau}v_p=\lambda_pv_p.
\end{equation}

Choose an arbitrary $a_2$ and $b_2$ such that $a_1>a_2>a$ and $b_1<b_2<b$. Let 
\[
\Omega_1=\Omega_{[a_2,b_2]}=\{x\in X : \Sigma_x\cap [a_2, b_2]=\emptyset\}.
\]
This is an open subset of $X$, which contains $\bar\Omega=\overline{\Omega_{[a,b]}}$. Moreover, since $B$ and $V$ are $C^\infty$-bounded, there exists $\epsilon>0$ such that  $\Omega_{1,2\epsilon}:=\{x\in \Omega_1 : d(x,\partial \Omega_1)>2\epsilon\}$ contains $\bar\Omega$.

Now let $\phi_p\in C^\infty_b(X)$ be supported in $\Omega_1$, $\phi_p\equiv 1$ on $\Omega_{1,\epsilon p^{-1/2}}$ and
\[
|\nabla\phi_p|<C_1p^{1/2}, \quad |\nabla^2\phi_p|<C_2p.
\] 
Then
\[
H_{p,\tau}(\phi_p v_p)=\lambda_p\phi_p v_p+[H_{p,\tau}, \phi_p] v_p.
\]
By Proposition~\ref{p:estimate}, 
there exist $C_0>0$ and $p_0\in\mathbb N$, such that for any $p>p_0$, we have 
\begin{equation}\label{e:lower-est}
\|(H_{p}-\lambda_p)\phi_p v_p\|\geq  C_0\|\phi_p v_p\|.
\end{equation}

\begin{lemma}\label{l:lower-est-a}
For any $\tau>0$ small enough,  there exists $C_0>0$, such that for any $p>p_0$, we have 
\begin{equation}\label{e:lower-est-a}
\|(H_{p,\tau}-\lambda_p)\phi_p v_p\|\geq  C_0\|\phi_p v_p\|.
\end{equation}
\end{lemma}

\begin{proof}
We can write
\begin{equation}\label{e:Hpa-lp}
\|(H_{p,\tau}-\lambda_p)\phi_p v_p\|\geq \|(H_{p}-\lambda_p)\phi_p v_p\|-\|(H_{p,\tau}-H_p)\phi_p v_p\|.
\end{equation}
By \eqref{e:DpaW} and \eqref{e:ApBp}, we have
\[
H_{p;\tau}-H_p=\frac{\tau}{\sqrt{p}} (-2d \Phi_p\cdot \nabla^{L^p\otimes E} +\Delta \Phi_p)-\tau^2|d \Phi_p|^2.
\]
Using \eqref{e:chi-p}, we get 
\begin{equation}\label{e:Hpa-Hp}
\|(H_{p;\tau}-H_p)\phi_p v_p \|\leq C_1\frac{\tau}{\sqrt{p}}\|\nabla^{L^p\otimes E}\phi_p v_p\| +C_2\tau\|\phi_p v_p \|+C_3\tau^2\|\phi_p v_p \|.
\end{equation}

To estimate the first term in the right hand side of \eqref{e:Hpa-Hp}, we proceed as follows:
\begin{align*}
\|\nabla^{L^p\otimes E}\phi_p v_p\|^2=& ((\nabla^{L^p\otimes E})^*\nabla^{L^p\otimes E}v_p, v_p)=(p(H_p-V)\phi_p v_p, \phi_p v_p)
\\
= & (p(H_p-\lambda_p)\phi_p v_p, \phi_p v_p)+p((\lambda_p-V)\phi_p v_p, \phi_p v_p)
\\
\leq & p\|(H_p-\lambda_p)\phi_p v_p\|\|\phi_p v_p\|+Cp\|\phi_p v_p\|^2
\\
\leq & \epsilon^2 p\|(H_p-\lambda_p)\phi_p v_p\|^2+(C+\epsilon^{-2}) p\|\phi_p v_p\|^2
\end{align*}
with an arbitrary $\epsilon>0$ to be chosen later, which gives an estimate
\[
\frac{1}{\sqrt{p}}\|\nabla^{L^p\otimes E}\phi_p v_p\|\leq \epsilon \|(H_p-\lambda_p)\phi_p v_p\|+(C+\epsilon^{-1}) \|\phi_p v_p\|. 
\]
Using this estimate, from \eqref{e:Hpa-Hp}, we get
\begin{multline*}
\|(H_{p;\tau}-H_p)\phi_p v_p \|\\ \leq C_1\tau\epsilon \|(H_p-\lambda_p)\phi_p v_p\|+(C_2+C_3\epsilon^{-1})\tau\|\phi_p v_p\| +C_4\tau^2\|\phi_p v_p \|.
\end{multline*}
We choose $\epsilon$ such that $C_1\tau\epsilon=\frac 12$:
\[
\|(H_{p;\tau}-H_p)\phi_p v_p \|\leq \frac 12 \|(H_p-\lambda_p)\phi_p v_p\|+C_5\tau\|\phi_p v_p \|+C_6\tau^2\|\phi_p v_p \|.
\]
Using \eqref{e:lower-est}, from this estimate and \eqref{e:Hpa-lp}, we get
\begin{multline*}
\|(H_{p,\tau}-\lambda_p)\phi_p v_p\|\geq \frac 12 \|(H_p-\lambda_p)\phi_p v_p\|-C_5\tau\|\phi_p v_p \|-C_6\tau^2\|\phi_p v_p \| \\
\geq \frac 12 C_0\|\phi_p v_p\|-C_5\tau\|\phi_p v_p \|-C_6\tau^2\|\phi_p v_p \|.
\end{multline*}
Taking $\tau$ small enough, we complete the proof.
\end{proof}

Let us fix $\tau$ as in Lemma \ref{l:lower-est-a}. By \eqref{e:Hpa-lambda}, we have
\begin{equation}\label{e:H=comm}
(H_{p,\tau}-\lambda_p)\phi_p v_p=[H_{p,\tau},\phi_p]v_p.
\end{equation}
By \eqref{e:DpaW} and \eqref{e:ApBp}, we compute
\begin{equation}\label{e:commHp-phi}
[H_{p,\tau}, \phi_p]=\frac 1p(-2d\phi_p\cdot\nabla^{L^p\otimes E} +\Delta\phi_p)-\frac{2\tau}{\sqrt{p}}d \Phi_p\cdot d \phi_p.
\end{equation}

Since  $d\phi_p$ and $\Delta\phi_p$ are supported in $\Omega_1\setminus \overline{\Omega_{1,\epsilon p^{-1/2}}}$, we have
\begin{equation}\label{e:est-psi}
\|d \phi_pv_p\|\leq Cp^{1/2}\|\psi_p v_p\|, \quad \|\Delta\phi_p\, v_p\|\leq Cp\| \psi_p v_p\|,
\end{equation}
where $\psi_p \in C^\infty_b(X)$ is supported in $X \setminus \overline{\Omega_{1,2\epsilon p^{-1/2}}}$ and $\psi_p\equiv 1$ on $X\setminus \Omega_{1,\epsilon p^{-1/2}}$, in particular on ${\rm supp}\,d\phi_p\subset \Omega_1 \setminus \overline{\Omega_{1, \epsilon p^{-1/2}}}$.

Therefore, by \eqref{e:commHp-phi}  and \eqref{e:est-psi}, we get 
\begin{equation}\label{e:commHp-phi-est}
\|[H_{p,\tau},\phi_p]v_p\|\leq\frac 2p\|d\phi_p\cdot \nabla^{L^p\otimes E} v_p\|+C_2\|\psi_p v_p\|.
\end{equation}

Now we need an estimate for $\|d\phi_p\cdot\nabla^{L^p\otimes E}v_p\|$, which is given by the following lemma.  

\begin{lem}\label{l:nabla-est}
We have 
\[
\|d\phi_p\cdot\nabla^{L^p\otimes E}v_p\|\leq Cp\|\psi_p v_p\|. 
\]
\end{lem}

The proof of this lemma will be given in Section \ref{s:nabla-est}. First, we complete the proof of Theorem \ref{t:eigenest}.

By \eqref{e:commHp-phi-est} and Lemma~\ref{l:nabla-est}, we infer that 
\[
\|[H_{p,\tau}, \phi_p]v_p\|\leq C_1\|\psi_p v_p\|.
\]
From this estimate, taking into account \eqref{e:lower-est} and \eqref{e:commHp-phi}, we get 
\[
\|\phi_p v_p\|\leq C_1\|\psi_p v_p\|.
\]
Now we proceed as follows: 
\begin{multline*}
\int_{\Omega} e^{2\tau\sqrt{p} \Phi_p(x)}|u_p(x)|^2dx \leq
\int_{1,\Omega_{\epsilon p^{-1/2}}} e^{2\tau\sqrt{p} \Phi_p(x)}|u_p(x)|^2dx \\ =\|v_p\|^2_{L^2(\Omega_{1, \epsilon p^{-1/2}})} \leq \|\phi_p v_p\|^2 \leq C^2_1 \|\psi_p v_p\|^2 = C^2_1 \|\psi_p v_p\|^2_{L^2(X \setminus \overline{\Omega_{1,2\epsilon p^{-1/2}}})}\\ \leq C^2_1 \|v_p\|^2_{L^2(X \setminus \overline{\Omega_{1, 2\epsilon p^{-1/2}}})}=C^2_1 \int_{X \setminus \overline{\Omega_{1, 2\epsilon p^{-1/2}}}} e^{2\tau\sqrt{p} \Phi_p(x)}|u_p(x)|^2dx\leq C^2_1, 
\end{multline*}
since $\Phi_p=0$ on $X \setminus \overline{\Omega_{1,2\epsilon p^{-1/2}}}\subset X \setminus\Omega$, that completes the proof of Theorem \ref{t:eigenest}. 

\subsection{Proof of Lemma~\ref{l:nabla-est}}\label{s:nabla-est}
As in Section \ref{s:sppr}, we choose an at most countable collection of coordinate charts (with $p=1$) 
\[
\varkappa_{\alpha}:=\varkappa_{x_{\alpha}}\left|_{B(0,c)}\right. : B(0,c) \to U_{\alpha} := \varkappa_{\alpha}(B(0,c))\subset X,  
\] 
with $1\leq \alpha\leq I$, $I\in \NN\cup \{\infty\}$, which cover $\Omega$ and for the cardinality of the set $\mathcal I_{\alpha}= \{1\leq \beta \leq I : U_{\alpha} \cap U_{\beta} \neq \varnothing \}$, we have 
\begin{equation}\label{e:Leb}
\# \mathcal I_{\alpha} \leq K_0, \quad 1\leq \alpha\leq I.
\end{equation}

Choose a family of smooth functions $\{\varphi_{\alpha} : \RR^{2n}\to [0; 1], 1\leq \alpha\leq I\}$ supported on the ball $B(0,c)$, which gives a partition of unity subordinate to $\{U_\alpha\}$:
\[
\sum_{\alpha=1}^{I}(\varphi_\alpha \circ \varkappa^{-1}_\alpha)^2 \equiv 1\ \text{on}\ \Omega,
\]
and satisfies the condition: for any $\gamma\in \ZZ^{2n}_+$, there exists $C_\gamma>0$ such that
\begin{equation}\label{e:unif}
|\partial^\gamma\varphi_\alpha(Z)|<C_\gamma, \quad Z\in \RR^{2n}, \quad 1\leq \alpha\leq I.
\end{equation}

Recall the localization formula
\[
\|Pu\|^2=\sum_{\alpha=1}^{I} \|P[(\varphi_\alpha \circ \varkappa^{-1}_\alpha)u]\|^2-\sum_{\alpha=1}^{I} \|[P,\varphi_\alpha \circ \varkappa^{-1}_\alpha] u\|^2,
\]
which gives
\begin{equation}\label{e:local3}
\|d\phi_p\cdot\nabla^{L^p\otimes E}v_p\|^2 
\leq \sum_{\alpha=1}^{I} \|d\phi_p\cdot\nabla^{L^p\otimes E}[(\varphi_\alpha \circ \varkappa^{-1}_\alpha)v_p]\|^2.
\end{equation}

For any $\alpha$, $1\leq \alpha\leq I$, choose a local orthonormal frame $(e^{(\alpha)}_1,\ldots,e^{(\alpha)}_{2n})$ in $TX$ defined on $U_\alpha$, which is $C^\infty$-bounded in $\alpha$. For simplicity of notation, we denote $u=(\varphi_\alpha \circ \varkappa^{-1}_\alpha)v_p\in C^\infty_c(U_\alpha)$ and will omit $\alpha$, writing $e_j$ instead of $e^{(\alpha)}_j$. We get
\[
d\phi_p\cdot\nabla^{L^p\otimes E}u=\sum_{j=1}^{2n}e_j\phi_p \nabla_{e_j}^{L^p\otimes E}u=\sum_{j=1}^{2n}\nabla_{e_j}^{L^p\otimes E}[(e_j\phi_p)u]-\sum_{j=1}^{2n} (e_je_j\phi_p) u
\]
and, therefore, by \eqref{e:est-psi}, we have
\begin{equation}\label{e:dphi:est}
\begin{aligned}
\|d\phi_p\cdot\nabla^{L^p\otimes E}u\|^2\leq & C_1\sum_{j=1}^{2n}\|\nabla_{e_j}^{L^p\otimes E}[(e_j\phi_p)u]\|^2+C_2p^2\|\psi_p u\|^2\\
\leq & C_1\sum_{j,k=1}^{2n}\|\nabla_{e_k}^{L^p\otimes E}[(e_j\phi_p)u]\|^2+C_2p^2\|\psi_p u\|^2\\
= & C_1\sum_{j=1}^{2n}\|\nabla^{L^p\otimes E}[(e_j\phi_p)u]\|^2+C_2p^2\|\psi_p u\|^2. 
\end{aligned}
\end{equation}
By \eqref{e:unif}, all the constants here and below can be taken to be independent of $\alpha$.

Now we apply the  formula
\[
\|P(\chi u)\|^2=\Re (Pu, P\chi^2u)+\|[P,\chi]u\|^2.
\]
For any $j=1,2,\ldots,2n$, using the fact that $[\nabla_{e_k}^{L^p\otimes E}, e_j\phi_p]=e_ke_j\phi_p$ and \eqref{e:est-psi}, we get 
\begin{align*}
\|\nabla^{L^p\otimes E}[(e_j\phi_p)u]\|^2=& \sum_{k=1}^{2n}\|\nabla_{e_k}^{L^p\otimes E}[(e_j\phi_p)u]\|^2\\ = & \sum_{k=1}^{2n} \Re (\nabla_{e_k}^{L^p\otimes E}u, \nabla_{e_k}^{L^p\otimes E}[(e_j\phi_p)^2u])+ \|(e_ke_j\phi_p) u\|^2\\
\leq & \sum_{k=1}^{2n} \Re ((\nabla_{e_k}^{L^p\otimes E})^*\nabla_{e_k}^{L^p\otimes E}u, (e_j\phi_p)^2u)+C_3p^2\|\psi_p u\|^2.
\end{align*}

The operator $H_p$ can be written as 
\[
H_p=\frac 1p(\sum_{k=1}^{2n}(\nabla_{e_k}^{L^p\otimes E})^*\nabla_{e_k}^{L^p\otimes E}-\nabla_{\sum_{k=1}^{2n}\nabla_{e_k}e_k}^{L^p\otimes E} )+V.
\]
Therefore, we proceed as follows:
\begin{multline}\label{e:nabla-e_j}
\|\nabla^{L^p\otimes E}[(e_j\phi_p)u]\|^2\\ \leq  p \Re (H_p u, (e_j\phi_p)^2u)+C\Re (\nabla_{\sum_{k=1}^{2n}\nabla_{e_k}e_k}^{L^p\otimes E}u, (e_j\phi_p)^2u)+C_3p^2\|\psi_p u\|^2.
\end{multline}

By \eqref{e:DpaW} and \eqref{e:ApBp}, 
%
%
we have 
\begin{multline*}
(H_pu,(e_j\phi_p)^2 u)=
\\
=((H_{p,\tau}+\frac{\tau}{\sqrt{p}}(2d\Phi_p\cdot \nabla^{L^p\otimes E}-\Delta \Phi_p)+\tau^2|d\Phi_p|^2)u,(e_j\phi_p)^2u).
\end{multline*}
Therefore, for the first term in the right hand side of \eqref{e:nabla-e_j}, we get 
\begin{multline}\label{e:nabla-e_j1}
|\Re (H_p u, (e_j\phi_p)^2u)|\leq |(H_{p,\tau}u,(e_j\phi_p)^2u)|\\ + \frac{C_1}{\sqrt{p}}(d \Phi_p\cdot \nabla^{L^p\otimes E}u,(e_j\phi_p)^2u)+C_2p\|\psi_p u\|^2.
\end{multline} 

For the first term in the right hand side of \eqref{e:nabla-e_j1}, we recall that $u=(\varphi_\alpha \circ \varkappa^{-1}_\alpha)v_p$ and use \eqref{e:Hpa-lambda}:
\begin{multline}\label{e:nabla-e_j1a}
(H_{p,\tau}u,(e_j\phi_p)^2u)=(H_{p,\tau}(\varphi_\alpha \circ \varkappa^{-1}_\alpha)v_p,(e_j\phi_p)^2u)
\\
=\lambda_p \|(e_j\phi_p)u\|^2+([H_{p,\tau}, \varphi_\alpha \circ \varkappa^{-1}_\alpha]v_p,(e_j\phi_p)^2u)\\
\leq ([H_{p,\tau}, \varphi_\alpha \circ \varkappa^{-1}_\alpha]v_p,(e_j\phi_p)^2u)+Cp\|\psi_p u\|^2.
\end{multline} 
Using the formula for the commutator $[H_{p,\tau}, \varphi_\alpha \circ \varkappa^{-1}_\alpha]$ (cf. \eqref{e:commHp-phi}), we get
\begin{multline*}
([H_{p,\tau}, \varphi_\alpha \circ \varkappa^{-1}_\alpha]v_p,(e_j\phi_p)^2u)\\
\begin{aligned}
= & \frac 1p((-2d(\varphi_\alpha \circ \varkappa^{-1}_\alpha) \cdot\nabla^{L^p\otimes E} +\Delta(\varphi_\alpha \circ \varkappa^{-1}_\alpha))v_p,(e_j\phi_p)^2u)\\
& -\frac{2\tau}{\sqrt{p}}(d \Phi_p\cdot d(\varphi_\alpha \circ \varkappa^{-1}_\alpha)v_p,(e_j\phi_p)^2u).
\end{aligned}
\end{multline*} 
By \eqref{e:est-psi}, it follows that
\begin{multline}\label{e:nabla-e_j11}
|([H_{p,\tau}, \varphi_\alpha \circ \varkappa^{-1}_\alpha]v_p,(e_j\phi_p)^2u)| \\
\leq \frac 2p |(d(\varphi_\alpha \circ \varkappa^{-1}_\alpha) \cdot\nabla^{L^p\otimes E}v_p,(e_j\phi_p)^2u)|
+C\sqrt{p}\|\psi_p v_p\|_{L^2(U_\alpha)}\|\psi_p u\|.
\end{multline} 
Here $\psi_p v_p$ is not supported in $U_\alpha$, but, since $d(\varphi_\alpha \circ \varkappa^{-1}_\alpha)$ is supported in $U_\alpha$, we can put the integration over $U_\alpha$. 

For the first term in the right hand side of \eqref{e:nabla-e_j11}, we proceed as follows:  
\begin{multline*}
(d(\varphi_\alpha \circ \varkappa^{-1}_\alpha) \cdot\nabla^{L^p\otimes E}v_p,(e_j\phi_p)^2u)\\
\begin{aligned}
= &\sum_{k=1}^{2n}(e_k(\varphi_\alpha \circ \varkappa^{-1}_\alpha)\nabla_{e_k}^{L^p\otimes E}v_p,(e_j\phi_p)^2u)\\
=& \sum_{k=1}^{2n}(e_k(\varphi_\alpha \circ \varkappa^{-1}_\alpha)(e_j\phi_p)\nabla_{e_k}^{L^p\otimes E}v_p,(e_j\phi_p)u)\\
=& \sum_{k=1}^{2n}(e_k(\varphi_\alpha \circ \varkappa^{-1}_\alpha)(e_j\phi_p)v_p,(\nabla_{e_k}^{L^p\otimes E})^*[(e_j\phi_p)u])
\\
& -\sum_{k=1}^{2n}(e_k[(e_k(\varphi_\alpha \circ \varkappa^{-1}_\alpha)(e_j\phi_p)]v_p,(e_j\phi_p)u).
\end{aligned}
\end{multline*}
Now we use the fact that $(\nabla_{e_k}^{L^p\otimes E})^*=-\nabla_{e_k}^{L^p\otimes E}+c_{k,p}$, where $c_{k,p}$ is an endomorphism of $L^p\otimes E$ over $U_\alpha$ such that $|c_{k,p}|=\mathcal O(p)$ and conclude that
\begin{multline*}
|(d(\varphi_\alpha \circ \varkappa^{-1}_\alpha) \cdot\nabla^{L^p\otimes E}v_p,(e_j\phi_p)^2u)|\\ \leq  Cp^{1/2} \|\psi_p u\|\|\nabla^{L^p\otimes E}[(e_j\phi_p)u]) +C_1p^{2}\|\psi_p v_p\|_{L^2(U_\alpha)}\|\psi_p u\|. 
\end{multline*}
Plugging this estimate into \eqref{e:nabla-e_j11}, we get 
\begin{multline}\label{e:nabla-e_j11a}
|([H_{p,\tau}, \varphi_\alpha \circ \varkappa^{-1}_\alpha]v_p,(e_j\phi_p)^2u)|  \\
\leq \frac{C}{\sqrt{p}}\|\psi_p u\|\|\nabla^{L^p\otimes E}[(e_j\phi_p)u])+Cp\|\psi_p v_p\|_{L^2(U_\alpha)}\|\psi_p u\|.
\end{multline} 
Combining \eqref{e:nabla-e_j1a} and \eqref{e:nabla-e_j11a}, we get an estimate for the first term in the right hand side of \eqref{e:nabla-e_j1},
\begin{multline}\label{e:nabla-e_j1b}
|(H_{p,\tau}u,(e_j\phi_p)^2u)|\\
\leq \frac{C}{\sqrt{p}}\|\psi_p u\|\|\nabla^{L^p\otimes E}[(e_j\phi_p)u]\|+Cp\|\psi_p v_p\|_{L^2(U_\alpha)}\|\psi_p u\|.
\end{multline}  

For the second term in the right hand side of \eqref{e:nabla-e_j1}, we proceed as follows:
\begin{multline*} 
 (d\Phi_p\cdot \nabla^{L^p\otimes E}u,(e_j\phi_p)^2u)=((e_j\phi_p)\sum_{k=1}^{2n}{e_k}\Phi_p \nabla^{L^p\otimes E}_{e_k}u,(e_j\phi_p)u)\\
 =(\sum_{k=1}^{2n}{e_k} \Phi_p \nabla^{L^p\otimes E}_{e_k}[(e_j\phi_p)u],(e_j\phi_p)u)-(\sum_{k=1}^{2n}{e_k}\Phi_p (e_ke_j\phi_p)u, (e_j\phi_p)u).
\end{multline*}
Therefore, we get 
\begin{multline} \label{e:nabla-e_j12}
|(d \Phi_p\cdot \nabla^{L^p\otimes E}u,(e_j\phi_p)^2u)|\\ \leq 
 C_1p^{1/2}\|\nabla^{L^p\otimes E}[(e_j\phi_p)u]\| \|\psi_p u\|+C_2p^{3/2}\|\psi_p u\|^2.
\end{multline}
Plugging \eqref{e:nabla-e_j1b} and \eqref{e:nabla-e_j12} into \eqref{e:nabla-e_j1}, we get an estimate for the first term in the right hand side of \eqref{e:nabla-e_j}
\begin{multline}\label{e:nabla-e_j1f}
|\Re (H_p u, (e_j\phi_p)^2u)|\\ \leq C_1\|\nabla^{L^p\otimes E}(e_j\phi_p)u\| \|\psi_p u\|+C_2p\|\psi_p v_p\|_{L^2(U_\alpha)}\|\psi_p u\|.
\end{multline} 

For the second term in the right hand side of \eqref{e:nabla-e_j}, we write
\begin{multline*}
(\nabla_{\sum_k\nabla_{e_k}e_k}^{L^p\otimes E}u, (e_j\phi_p)^2u)=((e_j\phi_p)\nabla_{\sum_k\nabla_{e_k}e_k}^{L^p\otimes E}u, (e_j\phi_p)u)\\
=(\nabla_{\sum_k\nabla_{e_k}e_k}^{L^p\otimes E}[(e_j\phi_p)u], (e_j\phi_p)u)+((\sum_k\nabla_{e_k}e_k) e_j\phi_p)u, (e_j\phi_p)u),
\end{multline*}
that gives an estimate
\begin{multline}\label{e:nabla-e_j2}
|\Re (\nabla_{\sum_k\nabla_{e_k}e_k}^{L^p\otimes E}u, (e_j\phi_p)^2u)|\\
\leq C_1p^{1/2}\|\nabla^{L^p\otimes E}[(e_j\phi_p)u]\| \|\psi_p u\|+C_2p^{3/2}\|\psi_p u\|^2.
\end{multline}
 
Plugging \eqref{e:nabla-e_j1f} and \eqref{e:nabla-e_j2} into \eqref{e:nabla-e_j}, we get an estimate 
\[
\|\nabla^{L^p\otimes E}[(e_j\phi_p)u]\|^2 \leq  C_1p\|\nabla^{L^p\otimes E}[(e_j\phi_p)u]\|\|\psi_p u\|+C_2p^2\|\psi_p v_p\|_{L^2(U_\alpha)}\|\psi_p u\|.
\] 
Now we proceed as follows: 
\begin{align*}
\|\nabla^{L^p\otimes E}[(e_j\phi_p)u]\|^2 \leq & \frac{C_1}{2}(\epsilon^{-1}p^2\|\psi_p u\|^2+\epsilon\|\nabla^{L^p\otimes E}[(e_j\phi_p)u]\|^2)\\ & +C_2p^2\|\psi_p v_p\|_{L^2(U_\alpha)}\|\psi_p u\|.
\end{align*} 
Taking $\frac{C_1}{2}\epsilon<\frac 12$ and using the fact that $\|\psi_p u\|\leq \|\psi_p v_p\|_{L^2(U_\alpha)}$, we infer that 
\[
\|\nabla^{L^p\otimes E}[(e_j\phi_p)u]\|\leq Cp\|\psi_p v_p\|_{L^2(U_\alpha)}. 
\]
Now we use \eqref{e:dphi:est} and recall that $u=(\varphi_\alpha \circ \varkappa^{-1}_\alpha)v_p$. We get
\[
\|d\phi_p\cdot\nabla^{L^p\otimes E}[(\varphi_\alpha \circ \varkappa^{-1}_\alpha)v_p]\|\leq Cp\|\psi_p v_p\|_{L^2(U_\alpha)}.
\]
By \eqref{e:local3}, it follows that
\[
\|d\phi_p\cdot \nabla^{L^p\otimes E}v_p\|^2 \leq C^2p^2\sum_{\alpha=1}^{I} \|\psi_p v_p\|_{L^2(U_\alpha)}^2.
\]
By \eqref{e:Leb}, we have the estimate
\[
\sum_{\alpha=1}^{I} \|\psi_p v_p\|^2_{L^2(U_\alpha)}\leq K_0 \|\psi_p v_p\|^2,
\]
which completes the proof of Lemma~\ref{l:nabla-est}.

\end{document}